\documentclass[twoside,a4,11pt]{amsart}

\usepackage[english]{babel}
\usepackage[utf8x]{inputenc}
\usepackage{amsmath}
\usepackage{graphicx}
\usepackage[colorinlistoftodos]{todonotes}
\usepackage{amsmath,amscd,amsthm}
\usepackage{amstext, amsfonts, a4}
\usepackage{amssymb}
\usepackage{tikz-cd}
\usepackage{fancyhdr}
\usepackage{enumerate}
\usepackage{cleveref}
\usepackage{xcolor}
\usepackage{verbatim}

\numberwithin{equation}{section}
\newtheorem{thm}[equation]{Theorem}

\newtheorem{prop}[equation]{Proposition}
\newtheorem{cor}[equation]{Corollary}
\newtheorem{lem}[equation]{Lemma}

\theoremstyle{definition}
\newtheorem{ex}[equation]{Example}

\newtheorem{qu}[equation]{Question}
\newtheorem{rem}[equation]{Remark}

\renewcommand{\dim}{\operatorname{\mathsf{dim}}}
\renewcommand{\deg}{\operatorname{\mathsf{deg}}}
\renewcommand{\sup}{\operatorname{\mathsf{sup}}}

\newcommand\ind{\operatorname{\mathsf{ind}}}
\renewcommand\exp{\operatorname{\mathsf{exp}}}

\newcommand\Br{\operatorname{\mathsf{Br}}}

\newcommand\cchar{\operatorname{\mathsf{char}}}

\newcommand\N{\operatorname{\mathsf{N}}}

\newcommand{\la}{\langle}
\newcommand{\ra}{\rangle}
\newcommand{\mg}[1]{{#1}^{\times}}
\newcommand{\sq}[1]{{#1}^{\times 2}}

\newcommand{\s}{\sigma}

\renewcommand{\setminus}{\smallsetminus}

\DeclareMathOperator*{\bigperp}{\raisebox{-.4ex}{\scalebox{1.4}{$\perp$}}}
\renewcommand{\bmod}{\,\,\mathsf{mod}\,\,}

\newcommand{\lla}{\la\!\la}
\newcommand{\rra}{\ra\!\ra}
\newcommand{\half}{\mbox{$\frac{1}2$}}
\newcommand{\I}{\mathsf{I}}
\newcommand{\W}{\mathsf{W}}
\newcommand{\sots}[1]{\mathsf{S}_2(#1)}

\DeclareMathOperator*{\newsum}{\raisebox{-0.3 ex}{\scalebox{1.3}{$\mathsf{\Sigma}$}}}
\renewcommand{\sum}{\newsum}

\renewcommand{\leq}{\leqslant}
\renewcommand{\geq}{\geqslant}
\renewcommand{\max}{\mathsf{max}}

\newcommand{\nat}{\mathbb{N}}

\newcommand{\follows}[2]{$(#1\Rightarrow #2)$\,}

\begin{document}
\title[A bound on the index of exponent-$4$ algebras]{A bound on the index of exponent-$4$ algebras in terms of the $u$-invariant}

\date{15.10.2023}
	
\author{Karim Johannes Becher}
\author{Fatma Kader B\.{i}ng\"{o}l}

\address{University of Antwerp, Department of Mathematics, Middelheim\-laan~1, 2020 Ant\-werpen, Belgium.}

\email{KarimJohannes.Becher@uantwerpen.be}
\email{FatmaKader.Bingol@uantwerpen.be}
	
\begin{abstract}
For a prime number $p$, an integer $e\geq 2$ and a field $F$ containing a primitive $p^e$-th root of unity, the index of central simple $F$-algebras of exponent $p^e$ is bounded in terms of the $p$-symbol length of $F$.
For a nonreal field $F$ of characteristic different from $2$, the index of central simple algebras of exponent $4$ is bounded in terms of the $u$-invariant of $F$.
Finally, a new construction for nonreal fields of $u$-invariant $6$ is presented.
	
\medskip\noindent
{\sc Classification} (MSC 2020): 12E15, 16K20, 16K50
 
\end{abstract}

\maketitle
	
\section{Introduction} 
Let $F$ be a field and $n$ a positive integer.
A central simple $F$-algebra of degree $n$ containing a subfield which is a cyclic extension of degree $n$ of $F$ is called \emph{cyclic} or a \emph{cyclic $F$-algebra}.
Given a cyclic field extension $K/F$ of degree $n$,  a generator $\s$ of its Galois group and  an element $b\in F^{\times}$, the rules
\begin{equation*}
    j^n=b\quad\text{and}\quad xj=j\sigma(x)\quad \text{for all}\quad x\in K
\end{equation*}
determine a multiplication on the $K$-vector space $K\oplus jK\oplus\ldots\oplus j^{n-1}K$ turning it into a cyclic $F$-algebra of degree $n$, which is denoted by $$[K/F,\sigma,b).$$   
Any cyclic $F$-algebra is isomorphic to an algebra of this form; see \cite[Theorem 5.9]{Alb39}.
Furthermore, any central $F$-division algebra of degree $2$ or $3$ is cyclic; see \cite[Theorem 11.5]{Alb39} for the degree-$3$ case.

Central simple $F$-algebras of degree $2$ are called \emph{quaternion algebras}.
We refer to \cite[p.~25]{BOI} for a discussion of quaternion algebras, including their standard presentation by symbols depending on two parameters from the base field.
If $\cchar F\neq2$, $a\in \mg{F}\setminus\sq{F}$ and $b\in\mg{F}$, then the $F$-quaternion algebra $(a,b)_F$ is equal to $[K/F,\s,b)$ for $K=F(\sqrt{a})$ and the nontrivial automorphism $\s$ of $K/F$.

We refer to \cite{Alb39} and \cite{Draxl} for the theory of central simple algebras, and to
\cite[Section 3]{ABGSV2011} for a survey on the role of cyclic algebras in this context.

Before we approach the problem in the focus of our interest, we fix some notation.
We set $\nat^+=\nat\setminus\{0\}$.
We denote by $\Br(F)$ the Brauer group of $F$, and for $n\in\nat^+$, we denote by $\Br_n(F)$ the $n$-torsion part of $\Br(F)$.
Let $p$ always denote a prime number. 

The following question was asked by Albert in \cite[p.126]{Albert1936} and is still open in general.

\begin{qu}\label{p-primary-torsion-gen-by-cyclic?}
For $n\in\nat^+$, is $\Br_{n}(F)$ generated by classes of cyclic algebras of degree dividing $n$?
\end{qu}

In view of the Primary Decomposition Theorem for central simple algebras (see e.g.~\cite[Corollary 9.11]{Draxl}), any such question can be reduced to the case where $n$ is a prime power.
Each of the following two famous results gives a positive answer to \Cref{p-primary-torsion-gen-by-cyclic?} under additional hypotheses on $F$ in relation to $n$.

\begin{thm}[Albert]\label{Albert-p-primary-torsion-Br-cyclic gen}
    Let $p$ be a prime number and assume that $\cchar F=p$. Let $e\in\nat^+$. 
    Then $\Br_{p^e}(F)$ is generated by classes of cyclic $F$-algebras of degree dividing $p^e$.
\end{thm}
\begin{proof}
    See \cite[Chapter VII, Section 9]{Alb39}.
\end{proof}

\begin{thm}[Merkurjev-Suslin]\label{Merkur-Suslin-n-torsion-Br-cyclic-gen}
    Let $n\in\nat^+$ and assume that $F$ contains a primitive $n$-th root of unity. 
    Then $\Br_n(F)$ is generated by the classes of cyclic $F$-algebras of degree dividing $n$.
\end{thm}
\begin{proof}
    See \cite{MS82}.
\end{proof}

If $F$ contains a primitive $n$-th root of unity then $\cchar F$ does not divide $n$.
Hence the hypotheses of \Cref{Albert-p-primary-torsion-Br-cyclic gen} and \Cref{Merkur-Suslin-n-torsion-Br-cyclic-gen} are mutually exclusive.

For $n=2$, \Cref{Merkur-Suslin-n-torsion-Br-cyclic-gen} was obtained by Merkurjev in \cite{Merk81}. 
Note that the hypothesis of \Cref{Merkur-Suslin-n-torsion-Br-cyclic-gen} for $n=2$ just means that $\cchar F\neq 2$.
Together with \Cref{Albert-p-primary-torsion-Br-cyclic gen} this gives an unconditional positive answer to \Cref{p-primary-torsion-gen-by-cyclic?} for $n=2$.

It was observed in \cite[Proposition 16.6]{MS82} that from the positive answer to \Cref{p-primary-torsion-gen-by-cyclic?} in the (highly nontrivial) case $n=2$ one obtains (rather easily) an unconditional positive answer for $n=4$. 
In \Cref{C:Br4genbycycles}, we obtain a different argument for this step.

Whenever we have a positive answer to \Cref{p-primary-torsion-gen-by-cyclic?}, it is motivated to look at quantitative aspects of the problem. In the first place, this concerns the number of cyclic algebras  needed for a tensor product representing a class in $\Br(F)$ of given exponent. 
This leads to the notion and the study of \emph{symbol lengths}.

For a central simple $F$-algebra $A$, the \emph{$n$-symbol length of} $A$, denoted by $\lambda_n(A)$, is the smallest $m\in\nat^+$ such that $A$ is Brauer equivalent to a tensor product of $m$ cyclic algebras of degree dividing $n$, if such an integer $m$ exists, otherwise we set $\lambda_n(A)=\infty$.
The \emph{$n$-symbol length of} $F$ is defined as 
$$\lambda_n(F)\,\,\,=\,\,\,\sup\{\lambda_n(A)\mid [A]\in\Br_n(F)\}\,\,\in\,\,\nat^+\cup\{\infty\}.$$
Note that the index of any central simple $F$-algebra of exponent $n$ is at most $n^{\lambda_n(F)}$.

Let $p$ be a prime number.
It seems plausible to take the $p$-symbol length of $F$ for a measure for the complexity of the whole $p$-primary part of the theory of central simple algebras over $F$.
So in particular one might expect that $\lambda_{p^e}(F)$ can be bounded in terms of $\lambda_p(F)$ for all $e\in\nat^+$. When $F$ contains a primitive $p^e$-th root of unity, it follows from \cite[Proposition 2.5]{Tig84} that $\lambda_{p^e}(F)\leq e\lambda_p(F)$, but 
in general, this problem is still open.

 In this article, we consider the following question.

\begin{qu}\label{exp-ind-rel-without-primitive root}
    Let $e\in\nat^+$. 
    Can one bound the index of a central simple $F$-algebra of exponent $p^e$ in terms of $e$ and $\lambda_p(F)$?
\end{qu}

This is obviously true when $e=1$.
In the case where $F$ contains a primitive $p^e$-th root of unity, one can distill from the proof of \cite[Proposition 2.5]{Tig84} an argument showing 
that the index of any central simple $F$-algebra of exponent $p^e$ is bounded by $p^{\frac{e(e+1)}2\lambda_p(F)}$.
We retrieve this bound in \Cref{T:p} by means of a lifting argument formulated in \Cref{P:cylce-lift}.

In \Cref{section: squares-in-brauer group}, we consider the case where $p^e=4$ and make no assumption on roots of unity.
For a nonreal field $F$, we obtain in \Cref{bounds-ind-u-inv} an upper bound on the index of exponent-$4$ algebras in terms of the $u$-invariant of $F$.

\Cref{section:u6} is devoted to the construction of examples of nonreal fields with given $u$-invariant  admitting a central simple algebra of given $2$-primary exponent and of comparatively large index; see \Cref{um-exp2^e-deg2^(me-1)}.
If $F$ is nonreal and $u(F)=4$, then by \Cref{bounds-ind-u-inv} the index of a central simple $F$-algebra of exponent $4$ is at most $8$, and we see in \Cref{examples-um-e-me-1} that this is optimal.
This example provides at the same time quadratic field extensions $K/F$ with $u(F)=4$ and $u(K)=6$; see \Cref{example-K/F-u6}.
Hence \Cref{section:u6} provides also an alternative construction of fields of $u$-invariant $6$.

\section{Multiplication by a power of $p$ in the Brauer group}\label{section:p-powers in Br(F)} 

For a finite field extension $K/F$, let $\N_{K/F}:K\to F$ denote the norm map.

\begin{thm}\label{Albert-extend-cyclic-when-p-th-root-inF}
    Let $\zeta\in F$ be a primitive $p$-th root of unity.
    Let $K/F$ be a cyclic field extension of degree $p^{e-1}$.
    Then $K/F$ embeds into a cyclic field extension of degree $p^{e}$ of $F$ if and only if $\zeta=\N_{K/F}(x)$ for some $x\in K$.
\end{thm}
\begin{proof}
    See \cite[Theorem 9.11]{Albert1937}.
\end{proof}

Let $A$ and $B$ be central simple $F$-algebras.
We write $A\sim B$ to indicate that $A$ and $B$ are Brauer equivalent.
For $n\in\nat^+$ we denote by $A^{\otimes n}$ the $n$-fold tensor product $A\otimes_F\ldots\otimes_FA$.

\begin{thm}[Albert]\label{Albert-powers of cyclic alg}
    Let $n,m\in\nat$ with $m\leq n$ and $b\in \mg F$. Let $L/F$ be a cyclic field extension of degree $p^n$ and let $\sigma$ be a generator of its Galois group.
    Let $K$ be the fixed field of $\s^{p^{n-m}}$ in $L$.
    Then $$[L/F,\sigma,b)^{\otimes p^{m}}\sim[K/F,\sigma|_K,b).$$
\end{thm}
\begin{proof}
    See \cite[Theorem 7.14]{Alb39}.
\end{proof}

\begin{cor}\label{lifting-cyclic-when-pth-root-in-F}
    Let $\zeta\in F$ be a primitive $p$-th root of unity.
    Let $e\in\nat^+$. For $\alpha\in\Br(F)$, the following are equivalent: 
\begin{enumerate}[$(i)$]
    \item $\alpha$ is the class of a cyclic $F$-algebra of degree $p^{e-1}$ containing a cyclic field extension $K/F$ of degree $p^{e-1}$ such that $\zeta=\N_{K/F}(x)$ for some $x\in K$.
    \item $\alpha=p\beta$ for the class $\beta\in\Br (F)$ of a cyclic $F$-algebra of degree $p^{e}$.
\end{enumerate}	
\end{cor}
\begin{proof}
    \follows{i}{ii}
    Assume that $K/F$ is  a cyclic field extension of degree $p^{e-1}$, $\s$ a generator of its Galois group and $b\in\mg F$ is such that $\alpha$ is represented by $[K/F,\sigma,b)$. 
    Assume further that $\zeta=\N_{K/F}(x)$ for some $x\in K$.
    By \Cref{Albert-extend-cyclic-when-p-th-root-inF}, there exists a field extension $L/K$ of degree $p$  such that $L/F$ is cyclic.
    Then $\s$ extends to an $F$-automorphism $\s'$ of $L$, and it follows that $\s'$ generates the Galois group of $L/F$.
    Let $\beta$ be the class of the cyclic $F$-algebra $[L/F,\sigma',b)$. 
    Since $[L:K]=p$ and $\s'|_K=\s$, we conclude by \Cref{Albert-powers of cyclic alg} that $p\beta=\alpha$. 

    \follows{ii}{i} Assume that $\alpha=p\beta$ where $\beta\in \Br (F)$ is the class of cyclic $F$-algebra of degree $p^{e}$. 
    Then $\beta$ is given by $[L/F,\sigma,b)$ for some cyclic field extension $L/F$ of degree $p^e$, a generator $\sigma$ of its Galois group and some $b\in F^{\times}$.
    Let $K$ denote the fixed field of $\s^{p^{e-1}}$ in $L$. 
    Then $K/F$ is cyclic of degree $p^{e-1}$, and we obtain by \Cref{Albert-extend-cyclic-when-p-th-root-inF} that $\zeta=\N_{K/F}(x)$ for some $x\in K$.
    By \Cref{Albert-powers of cyclic alg}, we have $[L/F,\sigma,b)^{\otimes p}\sim[K/F,\sigma|_K,b)$. 
    Hence $\alpha$ is given by $[K/F,\sigma|_K,b)$.
\end{proof}

Given a central simple $F$-algebra $A$, we denote by $\deg A$, $\ind A$ and $\exp A$, the degree, index and exponent of $A$, respectively.
For $\alpha\in\Br(F)$, we write $\ind\alpha$ and $\exp\alpha$ for the index and the exponent of any central simple $F$-algebra representing $\alpha$.

Given a field extension $F'/F$ and $\alpha\in\Br(F)$ we denote by $\alpha_{F'}$ the image of $\alpha$ under the natural map $\Br(F)\to \Br(F')$ induced by scalar extension.

Let $m\in\nat^+$. We call $\alpha\in\Br(F)$ an $m$-\emph{cycle} if $\exp\alpha=m=[K:F]$ for some cyclic field extension $K/F$ for which $\alpha_K=0$.
Hence, given a central $F$-division algebra $D$, the class of $D$ in $\Br(F)$ is an $m$-cycle if and only if $D$ is cyclic and $\exp D=\deg D=m$.
    
\begin{prop}\label{P:cylce-lift}
    Let $e,i\in\nat^+$ with $i\leq e$ and such that every cyclic field extension of degree $p^{i}$  of $F$ embeds into a cyclic field extension of degree $p^e$ of~$F$.
    Then every $p^i$-cycle in $\Br(F)$ is of the form $p^{e-i}\beta$ for a $p^e$-cycle $\beta\in\Br(F)$.
\end{prop}
\begin{proof}
    Let $\alpha\in\Br(F)$ be a $p^i$-cycle. Hence $\alpha$ is given by 
    $D=[K/F,\s,b)$ for a cyclic field extension $K/F$ of degree $p^i$, a generator $\s$ of its Galois group and some $b\in\mg F$. 
    In particular $\deg D=p^i=\exp \alpha=\exp D$, whereby $D$ is a division algebra.
    By the hypothesis, $K/F$ embeds into a cyclic field extension $L/F$ of degree $p^e$.
    Then $\s$ extends to an $F$-automorphism $\s'$ of $L$. It follows that $\s'$ is a generator of the Galois group of $L/F$. 
    We set $\Delta=[L/F,\s',b)$ and denote by $\beta$ the class of $\Delta$ in $\Br(F)$.
    We obtain by \Cref{Albert-powers of cyclic alg} that
    $\Delta^{\otimes p^{e-i}}\sim D$, whereby $p^{e-i}\beta=\alpha$. Since $\exp \alpha=p^i$, it follows that $\exp\beta=p^e=\deg\Delta$. 
    Since $\beta_L=0$, we conclude that $\beta$ is a $p^e$-cycle.
\end{proof}

An element $\alpha\in\Br(F)$ is called a \emph{cycle} if it is an $m$-cycle for some  $m\in\nat^+$ (given by $m=\exp \alpha$).

\begin{cor}\label{lifting-cyclic-deg p-when-p^e-th-root-inF}
    Let $e\in\nat^+$ be such that $F$ contains a primitive $p^e$-th root of unity.
    Then every cycle in $\Br_{p^e}(F)$ is a multiple of a $p^e$-cycle.
\end{cor}
\begin{proof}
    Let $\omega\in F$ be a primitive $p^e$-th root of unity and set $\zeta=\omega^{p^{e-1}}$. 
    Then $\zeta$ is a primitive $p$-th root of unity.
    For any field extension $K/F$ of degree $p^i$ with $1\leq i\leq e-1$, we have that $\zeta=(\omega^{p^{e-i-1}})^{p^{i}}=\N_{K/F}(\omega^{p^{e-i-1}})$. 
    Hence it follows by induction on $i$ from \Cref{Albert-extend-cyclic-when-p-th-root-inF} that every cyclic field extension of degree $p^i$ of $F$ embeds into a cyclic field extension of degree $p^e$. 
    Now the conclusion follows by \Cref{P:cylce-lift}.
\end{proof}

The following bound can be easily derived from the proof of \cite[Proposition 2.5]{Tig84}. To illustrate the general strategy, we include an argument.

\begin{thm}\label{T:p}
    Let $e\in\nat^+$ be such that $F$ contains a primitive $p^e$-th root of unity. 
    Then $\Br_{p^e}(F)$ is generated by the $p^e$-cycles.
    Furthermore,  for every $\alpha\in\Br_{p^e}(F)$, we have  $\ind\alpha=p^n$ for some $n\in\nat^+$ with $$n\leq \hbox{$ \frac{e(e+1)}{2}$}\lambda_p(F)\,.$$ 
\end{thm}
\begin{proof}
    Consider $\alpha\in\Br_{p^e}(F)$.
    By induction on $e$ we will show at the same time that $\alpha$ is a sum of $p^e$-cycles and 
    that $\ind\alpha$ is of the claimed form.
    
    We have $p^{e-1}\alpha\in\Br_{p}(F)$.
    It follows by \Cref{Merkur-Suslin-n-torsion-Br-cyclic-gen} for $n=p$ and by the definition of $\lambda_p(F)$ that $p^{e-1}\alpha=\sum_{i=1}^m \gamma_i$ for some
    natural number $m\leq \lambda_p(F)$ and classes $\gamma_1,\dots,\gamma_m\in \Br(F)$ of cyclic $F$-division algebras of degree $p$.
    Then $\gamma_1,\dots,\gamma_m$ are $p$-cycles.
    By \Cref{lifting-cyclic-deg p-when-p^e-th-root-inF}, for $1\leq i\leq m$, we have $\gamma_i=p^{e-1}\beta_i$ for a $p^e$-cycle $\beta_i\in\Br(F)$.
    
    We set  $\alpha'=\alpha-\sum_{i=1}^m\beta_i$. Then $\alpha'\in\Br_{p^{e-1}}(F)$.
    If $e=1$, then $\alpha'=0$ and $\alpha=\sum_{i=1}^m\beta_i$, and we obtain that $\ind\alpha=p^n$ for some positive integer $n\leq m\leq \lambda_p(F)$, confirming the claims about $\alpha$.
    Assume now that $e>1$. 
    By the induction hypothesis, $\alpha'$ is equal to a sum of $p^{e-1}$-cycles and $\ind\alpha'=p^{n'}$ for a natural number $n'\leq \frac{(e-1)e}{2}\lambda_p(F)$.
    By \Cref{lifting-cyclic-deg p-when-p^e-th-root-inF}, every cycle in $\Br_{p^e}(F)$ is a multiple of a $p^e$-cycle, hence in particular a sum of $p^e$-cycles.
    We conclude that $\alpha'$ is a sum of $p^e$-cycles, whereby $\alpha$ is a sum of $p^e$-cycles.
    Furthermore $\ind\alpha$ divides $\ind\alpha'\cdot \ind\beta_1\cdots\ind \beta_m=p^{n'+em}$.
    Hence $\ind\alpha=p^n$ for some positive integer  $$n\leq n'+em\leq \hbox{$\frac{(e-1)e}{2}$}\lambda_p(F)+e\lambda_p(F)=\hbox{$\frac{e(e+1)}{2}$}\lambda_p(F)\,.$$
    This proves the claims about $\alpha$.
\end{proof}

To obtain that $\Br_{p^e}(F)$ is generated by cycles, one can also conclude inductively on the basis of a weaker hypothesis on roots of unity than in \Cref{T:p}.

\begin{prop}\label{suff-condit-for-p-primary-gen-by-cyclic}
    Let $e\in\nat^+$ be such that $p\Br(F)\cap\Br_{p^{e-1}}(F)$ is generated by elements $p\beta$ with cycles $\beta\in\Br_{p^{e}}(F)$.
    Then  $\Br_{p^{e}}(F)$ is generated by cycles.
\end{prop}
\begin{proof}
    Consider $\alpha\in \Br_{p^e}(F)$.
    Then $p\alpha\in p\Br(F)\cap \Br_{p^{e-1}}(F)$, so the hypothesis implies that $p\alpha=\sum_{i=1}^{n}p\beta_i$ for some $n\in\nat$ and cycles $\beta_1,\ldots,\beta_n\in\Br_{p^{e}}(F)$. 
    Hence $\alpha-\sum_{i=1}^{n}\beta_i\in\Br_p(F)$. 
    By \Cref{Merkur-Suslin-n-torsion-Br-cyclic-gen}, $\alpha-\sum_{i=1}^{n}\beta_i=\sum_{i=1}^{m}\gamma_i$ for some $m\in\nat$ and $p$-cycles $\gamma_1,\ldots,\gamma_m\in \Br(F)$.
    Hence $\alpha$ is a sum of cycles in $\Br_{p^e}(F)$.
\end{proof}

\section{Multiplying by $2$ in the Brauer group}\label{section: squares-in-brauer group}

From now on we assume that $\cchar F\neq2$.
We show that the hypotheses of \Cref{suff-condit-for-p-primary-gen-by-cyclic} for $p=e=2$ are satisfied to retrieve the positive answer to \Cref{p-primary-torsion-gen-by-cyclic?} in the case where $p^e=4$.
The argument also yields bounds on the index of exponent-$4$ algebras in terms of the $2$-symbol length, and hence an affirmative answer to \Cref{exp-ind-rel-without-primitive root} for these algebras.

We denote by $\sots F$ the set of nonzero sums of two squares in $F$. Note that $\sots F$ is a subgroup of $F$.

The following statement is essentially contained in \cite[Corollary 5.14]{LLT1993}. We include the argument for convenience.
\begin{prop}\label{lifting-quaternion-nec suff cond}
    Let $Q$ be an $F$-quaternion division algebra. The following are equivalent:
\begin{enumerate}[$(i)$]
    \item $-1$ is a norm in a quadratic field extension of $F$ contained in $Q$.
    \item $-1$ is a reduced norm of $Q$.
    \item $Q\sim C^{\otimes2}$ for some cyclic $F$-algebra $C$ of degree $4$.
    \item $Q\simeq(s,t)_F$ for certain $s\in \sots F$ and $t\in F^{\times}$. 
\end{enumerate}	
\end{prop}
\begin{proof}
    Let $\mathsf{Nrd}_Q:Q\to F$ denote the reduced norm map.
    For any quadratic field extension $K/F$ contained in $Q$ and any $x\in K$ we have $\mathsf{Nrd}_Q(x)=\N_{K/F}(x)$. 
    Therefore the implication $(i \Rightarrow ii)$ is obvious, and for $(ii \Rightarrow i)$, it suffices to observe that, since $Q$ is a division algebra, every maximal commutative subring of $Q$ is a quadratic field extension of $F$.

    The equivalence $(i \Leftrightarrow iii)$ corresponds to the equivalence formulated in \Cref{lifting-cyclic-when-pth-root-in-F} in the case where $p=e=2$, taking for $\alpha\in\Br(F)$ the class of $Q$.

    To finish the proof, it suffices to show the equivalence $(i\Leftrightarrow iv)$.
    Since $\cchar{F}\neq 2$, any quadratic field extension of $F$ is of the form $F(\sqrt{s})$ for some $s\in \mg F\setminus \sq F$, and for such $s$, we have that  $-1$ is a norm in $F(\sqrt{s})/F$ if and only if the quadratic form $X^2+Y^2-sZ^2$ over $F$ is isotropic, if and only if  $s\in \sots F$.
    Finally, given a quadratic field extension $K/F$ contained in $Q$ and $s\in\mg{F}$ such that $K\simeq F(\sqrt{s})$, by \cite[Theorem 5.9]{Alb39}, we can find an element $t\in \mg F$ such that $Q\simeq (s,t)_F$. 
\end{proof}

We denote by $\W F$ the Witt ring of $F$ and by $\I F$ its fundamental ideal.
For $n\in\nat^+$, we set $\I^nF=(\I F)^n$, and we call a regular quadratic form over $F$ whose Witt equivalence class belongs to $\I^nF$ simply a \emph{form in $\I^nF$}.
Given a regular quadratic form $q$ over $F$, we denote by $\dim q$ its dimension (rank).
By a \emph{torsion form} we shall mean a regular quadratic form over $F$ whose class in $\W F$ has finite additive order.
A quadratic form $q$ such that $2\times q$ is hyperbolic is called a \emph{$2$-torsion form}.
The following statement describes $2$-torsion forms in $\I^2F$. 

\begin{lem}\label{beta decomp}
    Let $q$ be a form in $\I^2F$. Let $m\in\nat^+$ be such that $\dim q=2m+2$. 
    Then $2\times q$ is hyperbolic if and only if $q$ is Witt equivalent to $\bigperp_{i=1}^m a_i\lla  s_i,t_i\rra $ for some $s_1,\ldots,s_m\in \sots F$ and $a_1, t_1,\ldots,a_m, t_m\in  F^{\times}$.
\end{lem}
\begin{proof}
    For $s\in\sots F$ and $t\in F^{\times}$, the form $2\times\lla  s,t\rra $ over $F$ is hyperbolic. This proves the right-to-left implication.
 
    We prove the opposite implication by induction on $m$. 
    If $m=0$, then $q$ is a $2$-dimensional quadratic form in $\I^2 F$ and must therefore be hyperbolic. In particular, the statement holds in this case.
    Suppose now that $m\geq1$.
    In view of the induction hypothesis, we may assume without loss of generality that $q$ is anisotropic. 
    As the quadratic form $2\times q$ is hyperbolic and hence in particular isotropic, it follows by \cite[Lemma 6.24]{EKM} that $q\simeq q_1\perp q_2$ for certain regular quadratic forms $q_1$ and $q_2$ over $F$ such that $\dim q_1=2$ and  $2\times q_1$ is hyperbolic.
    We fix an element $a_1\in\mg F$ represented by $q_1$.
    Then $q_1\simeq\langle a_1,-a_1s_1\rangle$  for some $s_1\in\mg F$.
    As $2\times q_1$ is hyperbolic, so is $2\times \la 1,-s_1\ra$, whereby $s_1\in\sots F$.
    We write $q_2\simeq \la a\ra\perp q'$ with $a\in\mg F$ and a $(2m-1)$-dimensional regular quadratic form $q'$ over $F$.  
    We set $q''=q'\perp\langle s_1a\rangle$ and $t_1=-a_1a$.
    We obtain that $q\perp -q''$ is Witt equivalent to $a_1\lla  s_1,t_1\rra $. 
    Since $s_1\in\sots F$, we have that $2\times \lla s_1,t_1\rra$ is hyperbolic.
    Therefore $2\times q''$ is Witt equivalent to $2\times q$, and hence equally hyperbolic.
    Furthermore, $q''$ is a form in $\I^2 F$. Since $\dim q''=2m$ and $2\times q''$ is hyperbolic,  the induction hypothesis yields that there exist $s_2,\ldots,s_m\in\sots F$ and $a_2,t_2,\ldots,a_m,t_m\in F^{\times}$ such that $q''$ is Witt equivalent to $\bigperp_{i=2}^m a_i\lla  s_i,t_i\rra $.
    Then $q$ is Witt equivalent to $\bigperp_{i=1}^m a_i\lla  s_i,t_i\rra $. 
    This concludes the proof.
\end{proof}

By \cite[Theorem 14.3]{EKM}, associating to a quadratic form  its Clifford algebra induces a homomorphism
$$e_2: \I^2F\to\Br_2(F)\,.$$ 
By Merkurjev's Theorem \cite[Theorem 44.1]{EKM} together with \cite[Theorem 4.1]{Mil70}, the kernel of this homomorphism is precisely $\I^3 F$.

For a quadratic field extension $K/F$, we denote by $\mathsf{cor}_{K/F}$ the corestriction homomorphism $\Br(K)\to\Br(F)$ defined in  \cite[Section 3.B]{BOI} (where it is denoted by $\mathsf{N}_{K/F}$).

\begin{prop}\label{squre-Br-liftable quaternions}
    Let $\beta\in\Br_2(F)$. The following are equivalent:
\begin{enumerate}[$(i)$]
    \item $\beta \in 2\Br(F)$.
    \item $\beta = e_2(q)$ for some $2$-torsion  form $q$ in $\I^2F$.
    \item $\beta $ is given by $\bigotimes_{i=1}^m(s_i,t_i)_F$ for certain $m\in\nat$, $s_1,\ldots,s_m\in\sots F$ and $t_1,\ldots,t_m\in \mg F$.
\end{enumerate} 
    Moreover, if these conditions are satisfied and $\ind \beta \leq 4$, then one can choose $m$ in $(iii)$ such that $\ind \beta=2^m$.
\end{prop}
\begin{proof}
    The implication $(iii \Rightarrow i)$ follows by \Cref{lifting-quaternion-nec suff cond}.
	
    For $m\in\nat$, $s_1,\ldots,s_m\in \sots F$ and $a_1,t_1,\ldots,a_m,t_m\in F^{\times}$, one has that $e_2(\bigperp_{i=1}^ma_i\lla  s_i,t_i\rra )\sim\bigotimes_{i=1}^m(s_i,t_i)_F$.
    Hence the equivalence $(ii\Leftrightarrow iii)$ follows by \Cref{beta decomp}.
	
    We show now the implication $(i\Rightarrow iii)$.
    If $-1\in F^{\times2}$, then  $\mg F = \sots F$, so $(iii)$ holds by \Cref{Merkur-Suslin-n-torsion-Br-cyclic-gen}. 
    Assume now that $-1\in F^{\times}\setminus F^{\times2}$ and that $(i)$ holds. We set $K=F(\sqrt{-1})$.
    As $\beta \in\Br_2(F)$, it follows by \Cref{Merkur-Suslin-n-torsion-Br-cyclic-gen} together with \cite[Corollary A4]{LLT1993} that $\beta \cup(-1)=0$ in $H^3(F,\mu_2)$. 
    By \cite[Theorem 99.13]{EKM}, we obtain that $\beta =\mathsf{cor}_{K/F}\beta '$ for some $\beta '\in\Br_2(K)$.
    By \Cref{Merkur-Suslin-n-torsion-Br-cyclic-gen} and \cite[Proposition 100.2]{EKM}, $\Br_2(K)$ is generated by the classes of $K$-quaternion algebras $(x,t)_K$ with $x\in\mg K$ and $t\in \mg F$, and the corestriction with respect to $K/F$ of such a class is given by $(\N_{K/F}(x),t)_F$.
    Since $\N_{K/F}(\mg{K})\subseteq \sots F$ and $\beta  =\mathsf{cor}_{K/F} \beta '$, we obtain
    that $\beta $ is given by $\bigotimes_{i=1}^m(s_i,t_i)_F$ for some $m\in\nat$, $s_1,\dots,s_m\in\sots F$ and $t_1,\dots,t_m\in \mg F$.

    Hence the equivalence of $(i)$--$(iii)$ is established and it remains to prove the supplementary statement under the assumption that $\ind \beta \leq4$. In this case $\beta $ is the class of an $F$-biquaternion algebra.
    It follows by \cite[Section 16.A]{BOI} that $\beta = e_2(q')$ for a $6$-dimensional form $q'$ in $\I^2F$.
    By $(ii)$, there also exists  a $2$-torsion form $q$ in $\I^2 F$ with $\beta =e_2(q)$.
    Then $q'\perp -q$ is a form in $\I^2 F$ with $e_2(q'\perp-q)=0$. As mentioned above, this implies that $q'\perp-q$ is a form in $\I^3F$.
    Since $2\times q$ is hyperbolic, the Witt class of $2\times q'$ lies in $\I^4F$. 
    Note that $\dim2\times q'<16$. 
    Thus $2\times q'$ is hyperbolic, by \cite[Theorem 23.7]{EKM}, and hence \Cref{beta decomp} yields the result. 
\end{proof}

By \Cref{squre-Br-liftable quaternions}, for $p=e=2$, the hypotheses of \Cref{suff-condit-for-p-primary-gen-by-cyclic} on $2\Br(F)\cap\Br_2(F)$ are satisfied unconditionally. 
Hence one gets a positive answer to \Cref{p-primary-torsion-gen-by-cyclic?} for $p^e=4$. 
We will formulate this result together with a bound on the index of exponent-$4$ algebras in terms of the $2$-symbol length.

For $\alpha\in\Br_4(F)$, we denote by
$\mu(\alpha)$ the smallest  $m\in\nat$ for which there exist $s_1,\ldots,s_m\in\sots F$ and $t_1,\ldots,t_m\in F^{\times}$ with $2\alpha=\sum_{i=1}^{m}[(s_i,t_i)_F]$, noticing that such a representation does exist in view of \Cref{squre-Br-liftable quaternions}.
We set further 
 $$\mu(F)=\sup\,\{\mu(\alpha)\mid \alpha\in\Br_4(F)\}\in\nat\cup\{\infty\}.$$
\begin{rem} 
    If  $\sots F=F^{\times}$, then $\mu(F)=\lambda_2(F)$.
\end{rem}

The invariants $\lambda_2(F)$ and $\mu(F)$ are related to the existence of anisotropic torsion (respectively $2$-torsion) forms over $F$ in certain dimensions.
Recall that the $u$-invariant of $F$ is defined as 
$$u(F)=\sup\{\dim q\mid q\,\text{anisotropic torsion form over}\, F\}\in\nat\cup\{\infty\}.$$

We  refer to \cite[Chapter~8]{Pfi95} for a general discussion of this invariant.
\begin{prop}\label{2-symbol-u-inv}
    If $F$ is nonreal, then $\lambda_2(F)\leq\max\left\{0,\half u(F)-1\right\}$.
\end{prop} 
\begin{proof}
    See \cite[Théorème 2]{Kahn1990}.
\end{proof}

In  \cite[Section 8.2]{Pfi95}, the following relative of the $u$-invariant is studied. 

\begin{equation*}
    \mbox{$u'(F)=\sup\left\{\dim q\mid  q \text{ anisotropic $2$-torsion form over $F$}\right\}\in\nat\cup\{\infty\}$.}
\end{equation*}
Note that clearly $u'(F)\leq u(F)$.

\begin{prop}\label{rel-mu-u}
    We have $\mu(F)\leq \max\left\{0,\half u'(F)-1\right\}\,.$
\end{prop}

\begin{proof}
    We need to show that $\mu(\alpha)\leq m$ holds for any $\alpha\in\Br_4(F)$ and any $m\in\nat^+$ with $u'(F)\leq 2m+2$.
    Let $m\in\nat^+$ be such that $u'(F)\leq 2m+2$. Let $\alpha\in\Br_4(F)$. 
    By \Cref{squre-Br-liftable quaternions}, we have $2\alpha=e_2(q)$ for some $2$-torsion form $q$ in $\I^2F$. 
    Then $\dim q \leq u'(F)\leq2m+2$. 
    Hence $q$ is even-dimensional and we obtain that $q$ is Witt equivalent to a quadratic form of dimension $2m+2$.
    It follows by \Cref{beta decomp} that $q$ is Witt equivalent to $\bigperp_{i=1}^m a_i\lla  s_i,t_i\rra $ for some $s_1,\ldots,s_m\in\sots F$ and $a_1, t_1,\ldots,a_m,t_m\in F^{\times}$. 
    Then $$2\alpha=e_2(q)=e_2\left(\bigperp_{i=1}^m a_i\lla  s_i,t_i\rra \right)=\sum_{i=1}^{m}[(s_i,t_i)]\,,$$ whereby $\mu(\alpha)\leq m$. 
\end{proof}

The last statements motivate the following question.

\begin{qu}\label{Q:mu-lambda2}
    Is $\mu(F)\leq \lambda_2(F)$?
\end{qu}    

If $\lambda_2(F)\leq 2$, then a positive answer to \Cref{Q:mu-lambda2} is obtained by  \Cref{squre-Br-liftable quaternions}.
In the following example, the inequality in \Cref{rel-mu-u} is strict.

\begin{ex}
    Consider the iterated power series field $F=\mathbb{C}(\!(x)\!)(\!(y)\!)(\!(z)\!)$.
    The $8$-dimensional quadratic form $\varphi=\la 1,x,y,z,xy,xz,yz,xyz\ra$ over $F$ is anisotropic. 
    Since $-1$ is square in $F$ and $\mg{F}/\sq{F}$ is generated by the square-classes of $x,y$ and $z$, it is easy to see that every anisotropic quadratic form over $F$ is a subform of $\varphi$.
    This implies on the one hand that $u(F)=8$, on the other hand that $\lambda_2(F)=1$, because there is no anisotropic $6$-dimensional form in $\I^2 F$. 
    Furthermore $-1\in\sq{F}$, so $u'(F)=u(F)=8$ and $\mu(F)=\lambda_2(F)=1$.
\end{ex}

\begin{prop}\label{P:Br4genmodBr2by4cylces}
    Let $\alpha\in\Br_{4}(F)$. There exist a natural number $m\leq\mu(F)$ and $4$-cycles $\alpha_1,\ldots,\alpha_m\in\Br(F)$ such that $\alpha\equiv\sum_{i=1}^{m}\alpha_i\bmod {\Br_{2}(F)}$.
\end{prop}
\begin{proof}
    By \Cref{squre-Br-liftable quaternions} and the definition of $\mu(F)$, there exist a natural number $m\leq\mu(F)$ and $s_1,\ldots,s_m\in\sots F$ and $t_1,\ldots,t_m\in F^{\times}$ such that $2\alpha=\sum_{i=1}^{m}[(s_i,t_i)_F]$.
    By \Cref{lifting-quaternion-nec suff cond}, for $1\leq i\leq m$, we can find a $4$-cycle $\alpha_i\in\Br_4(F)$ such that $2\alpha_i=[(s_i,t_i)_F]$.
    We obtain that $2\alpha-\sum_{i=1}^{m}2\alpha_i=0$, whereby $\alpha-\sum_{i=1}^{m}\alpha_i\in\Br_2(F)$.
    Therefore $\alpha\equiv\sum_{i=1}^{m}\alpha_i\bmod {\Br_{2}(F)}$.	
\end{proof}

We retrieve \cite[Proposition 6.16]{MS82}:

\begin{cor}\label{C:Br4genbycycles}
    $\Br_4(F)$ is generated by cycles.
\end{cor}
\begin{proof}
    By \Cref{Merkur-Suslin-n-torsion-Br-cyclic-gen}, $\Br_2(F)$ is generated by classes of $F$-quaternion division algebras and thus by $2$-cycles.
    The statement now follows by combining this fact with \Cref{P:Br4genmodBr2by4cylces}.
\end{proof}

\begin{thm}\label{Br_4-bound-mixed-symbol-length}
    We have $\lambda_4(F)\leq \lambda_2(F)+\mu(F)$.  
    Furthermore, for $\alpha\in\Br(F)$, there exist $\beta\in\Br_4(F)$ with $\lambda_4(\beta)\leq\mu(\alpha)$ and $\gamma\in\Br_2(F)$ such that $\alpha=\beta+\gamma$, and in particular 
    $\ind \alpha =2^n$ for some natural number  $n\leq \lambda_2(F)+2\mu(F)$.
\end{thm}
\begin{proof}
    Let $\alpha\in\Br(F)$ and set $m=\mu(\alpha)$.
    By \Cref{P:Br4genmodBr2by4cylces}, we obtain that $\alpha=\sum_{i=1}^{m}\alpha_i+\gamma$ for some $4$-cycles $\alpha_1,\ldots,\alpha_m\in\Br_{4}(F)$ and some $\gamma\in\Br_{2}(F)$.
    Set $\beta=\sum_{i=1}^{m}\alpha_i$. Then $\beta\in\Br_{4}(F)$ and 
    $$\lambda_4(\alpha)\leq \lambda_4(\gamma)+\lambda_4(\beta)\leq \lambda_2(\gamma)+m\leq \lambda_2(F)+\mu(F)\,.$$
      
    Note that $\ind\beta$ divides $\prod_{i=1}^{m}\ind\alpha_i=2^{2m}$.
    Since $\ind\gamma$ divides $2^{\lambda_2(\gamma)}$ and $\ind\alpha$ divides $\ind\beta\cdot\ind\gamma$, we obtain that $\ind\alpha=2^n$ for some $n\in\nat$ with $n\leq \lambda_2(F)+2\mu(F)$.
\end{proof}

Note that when $F$ contains a primitive $4$-th root of unity, the bounds in \Cref{Br_4-bound-mixed-symbol-length} coincide with those in \Cref{T:p}.

\begin{cor}\label{bounds-ind-u-inv}
    Assume that $F$ is nonreal. Let $\alpha\in\Br_4(F)$. 
    Then $\ind\alpha=2^n$ for some natural number  $n\leq \max\left\{0,3\left(\half u(F)-1\right)\right\}$.
\end{cor}
\begin{proof}
    Since $u'(F)\leq u(F)$, this follows by \Cref{Br_4-bound-mixed-symbol-length} together with \Cref{2-symbol-u-inv} and \Cref{rel-mu-u}.
\end{proof}

\begin{prop}\label{rel-decom-symbol-mu}
    Let $l=\lambda_2(F)$ and $m=\mu(F)$ and assume that $l+m<\infty$.
    Let $D$ be a central $F$-division algebra of degree $2^{l+2m}$ for which $D^{\otimes 4}$ is split.
    There exist $F$-quaternion algebras $Q_1,\dots,Q_l$ and cyclic $F$-algebras $C_1,\dots,C_m$ of degree $4$ such that $$D\simeq \left(\bigotimes_{i=1}^l Q_i\right)\otimes\left(\bigotimes_{i=1}^m C_i\right)\,.$$
\end{prop}
\begin{proof}
    By \Cref{Br_4-bound-mixed-symbol-length}, the class of $D$ in $\Br (F)$ is represented by such a tensor product, and since the degrees coincide, the statement follows.
\end{proof}

\begin{cor}\label{rel-decom-u-inv}
    Assume that $F$ is nonreal and let $m\in\nat$ be such that \linebreak $u(F)=2m+2$. Let $D$ be a central $F$-division algebra such that $D^{\otimes 4}$ is split and $\deg D=2^{3m}$.
    Then $D$ is decomposable into a tensor product of $m$ $F$-quaternion algebras and $m$ cyclic $F$-algebras of degree $4$.
\end{cor}
\begin{proof}
    Since $u(F)=2m+2$, we have $\lambda_2(F)\leq m$, by \Cref{2-symbol-u-inv}, and further $\mu(F)\leq m$, by \Cref{rel-mu-u}.
    The statement follows by \Cref{rel-decom-symbol-mu}.
\end{proof}

\begin{thm}\label{u-4-quadext-6-criterion}
    Assume that $F$ is nonreal with $u(F)=4$. Let $D$ be a central $F$-division algebra of degree $8$ such that $D^{\otimes 4}$ is split. 
    Then $D$ decomposes into a tensor product of a cyclic $F$-algebra of degree $4$ and an $F$-quaternion algebra. 
    Furthermore, $\ind D^{\otimes 2}=2$, and $u(K)=6$ holds for every quadratic field extension $K/F$ such that $(D^{\otimes 2})_K$ is split.
\end{thm}
\begin{proof}
    The first part follows by \Cref{rel-decom-u-inv} applied with $m=1$. 
    
    Since $u(F)=4$, we have $\lambda_2(F)\leq 1$, by \Cref{2-symbol-u-inv}.
    Hence $\ind C\leq 2$ for every central simple $F$-algebra $C$ such that $C^{\otimes 2}$ is split.
    Since $\ind D>2$ and $D^{\otimes 4}$ is split, we conclude that $\ind D^{\otimes 2}=2$.

    Consider now a quadratic field extension $K/F$ such that $(D^{\otimes 2})_K$ is split.
    Note that $(D^{\otimes 2})_K\simeq(D_K)^{\otimes 2}$ and $\ind D_K\geq \frac{1}2 \ind D=4$.
    Hence $D_K$ represents an element of $\Br_2(K)$ which is not given by any $K$-quaternion algebra. This shows that $\lambda_2(K)\geq 2$.
    It follows by \Cref{2-symbol-u-inv} that $u(K)\geq 6$.
    On the other hand,  since $u(F)=4$ and $[K:F]=2$, it follows by \cite[Theorem 4.3]{EL1973} that $u(K)\leq \frac{3}2 u(F)\leq 6$.
    Therefore $u(K)=6$.
\end{proof}

\section{Examples of fields with $u$-invariant $6$}
\label{section:u6}

In this section, we provide a construction leading to an example which shows that the bound in \Cref{bounds-ind-u-inv} is optimal for fields of $u$-invariant $4$.
In particular this construction provides examples of nonreal fields of $u$-invariant~$6$.

Let $q$ be a regular quadratic form over $F$ of dimension $n\geq2$. If $n=2$, then assume that $q$ is not hyperbolic. 
Then as a polynomial in $F[X_1,\ldots,X_n]$, the quadratic form $q(X_1,\ldots,X_n)$ is irreducible. 
Thus the ideal generated by $q(X_1,\ldots,X_n)$ in the polynomial ring $F[X_1,\ldots,X_n]$ is a prime ideal, and hence the quotient ring $F[X_1,\ldots,X_n]/(q(X_1,\ldots,X_n))$ is a domain. 
Its fraction field is denoted by $F(q)$ and called the \emph{function field of} $q$ \emph{over} $F$.

\begin{lem}\label{division-F(q)-dimq=2m+1}
    Let $m,n\in\nat^+$. Let $\alpha\in \Br(F)$ be such that $\ind \alpha=2^n$.
    Let $q$ be a regular $(2m+1)$-dimensional 
    quadratic form over $F$ such that $\ind \alpha_{F(q)}<\ind \alpha$.
    Then $n\geq m$. Moreover, if $n>m$, then $\ind 2\alpha\leq 2^{n-m-1}$.
\end{lem}
\begin{proof}
    Let $D$ be the central $F$-division algebra representing $\alpha$ in $\Br(F)$.
    Then $\deg D=\ind \alpha=2^n$.
    Let $C_0(q)$ denote the even Clifford algebra of $q$.
    By \cite[Proposition ~11.6]{EKM}, the $F$-algebra $C_0(q)$ is central simple.
    As $\dim_F C_0(q)=2^{2m}$, we have $\deg C_0(q)=2^m$. 
    By \cite[Example 11.3 and Proposition ~11.4~$(b)$]{EKM}, $C_0(q)$ carries an $F$-linear involution. Therefore  $(C_0(q))^{\otimes 2}$ is split.
    
    Since $\ind D_{F(q)}= \ind \alpha_{F(q)}<\ind \alpha=\deg D$, 
    it follows by \cite[Proposition ~30.5]{EKM}, that there exists an $F$-algebra homomorphism $C_0(q)\to D$.
    As $C_0(q)$ and $D$ are central simple $F$-algebras, it follows that $D\simeq C_0(q)\otimes_F B$ for a central  $F$-division algebra $B$. 
    Hence $2^n=\deg D = 2^m\cdot \deg B$, so in particular $n\geq m$. 
    
    Assume now that $n>m$. Then $\ind B=\deg B=2^{n-m}\geq2$.
    Since $(C_0(q))^{\otimes 2}$ is split, the class $2\alpha\in\Br_2(F)$ is given by $B^{\otimes 2}$. 
    Hence  $\ind 2\alpha=\ind B^{\otimes2}$.
    By \cite[Lemma 5.7]{Alb39}, we have  $\ind B^{\otimes2}\leq\frac{1}{2}\ind B$. Therefore  $\ind 2\alpha\leq 2^{n-m-1}$. 
\end{proof}

\begin{thm}\label{zorn-merkur}
    Let $\mathcal{C}$ be a class of field extensions of $F$ with the following properties:
\begin{enumerate}[$(i)$]
    \item $\mathcal{C}$ is closed under direct limits,
    \item if $L/F\in\mathcal{C}$ and $K/F$ is a subextension of $L/F$ then $K/F\in\mathcal{C}$,
    \item $F/F\in\mathcal{C}$.
\end{enumerate}
    Then there exists a field extension $K/F\in\mathcal{C}$ such that $K(\varphi)/F\notin\mathcal{C}$ for any anisotropic quadratic form $\varphi$ over $K$ of dimension at least $2$. 
\end{thm}
\begin{proof}
    See \cite[Theorem 6.1]{Bec2004}.
\end{proof}

The following statement and its hypotheses are motivated by an application which we obtain in \Cref{examples-um-e-me-1}.
\begin{prop}\label{um-exp2^e-deg2^(me-1)}
    Let $m,e\in\nat^+$ with $m\geq2$.
    Let $\alpha\in\Br(F)$ be such that $\exp \alpha=2^e$, $\ind \alpha=2^{me-1}$ and $\ind 2^{i} \alpha=2^{me-1-i}$ for $0\leq i\leq e-1$.
    There exists a field extension $K/F$ such that $u(K)\leq2m$, $\exp \alpha_K=2^e$ and $\ind \alpha_K=2^{me-1}$.
\end{prop}
\begin{proof}
    Let $\mathcal{C}$ be the class of field extensions $K/F$ such that $\ind 2^i\alpha_K\geq2^{me-mi-1}$ for $0\leq i\leq e-1$.
    Then $\mathcal{C}$ satisfies the conditions of \Cref{zorn-merkur}. 
    Hence there exists a field extension $K/F\in\mathcal{C}$ such that $K(\varphi)/F\notin\mathcal{C}$ for any anisotropic quadratic form $\varphi$ over $K$ of dimension at least $2$.
    As $\ind 2^{e-1}\alpha_K\geq 2^{m-1}$, $m\geq2$ and $\exp \alpha=2^e$, we get that $\exp \alpha_K=2^e$. 
    Since $\ind \alpha_K\geq2^{me-1}$ and $\ind \alpha=2^{me-1}$, we have that $\ind \alpha_K=2^{me-1}$.
	
    Let $\varphi$ be an arbitrary $(2m+1)$-dimensional quadratic form over $K$. We claim that $\varphi$ is isotropic.
    Let $\alpha_i=2^i\alpha_K$ for $0\leq i\leq e-1$. 
    We will check for $0\leq i\leq e-1$ that the inequality $\ind \alpha_i\geq 2^{me-mi-1}$ is preserved under scalar extension from $K$ to $K(\varphi)$.
    Consider first the case where $i=e-1$.
    If $\ind \alpha_{e-1}=2^{m-1}$, then $\ind (\alpha_{e-1})_{K(\varphi)}=2^{m-1}$, by  \Cref{division-F(q)-dimq=2m+1}.
    Otherwise, $\ind \alpha_{e-1}\geq2^{m}$, and therefore $\ind(\alpha_{e-1})_{K(\varphi)}\geq2^{m-1}$.
    Consider now the case where $0\leq i\leq e-2$.
    Note that $me-mi-1\geq m+1$, because $m\geq2$.
    If $\ind \alpha_{i}=2^{me-mi-1}$, then since $\ind 2\alpha_i=\ind \alpha_{i+1}\geq2^{me-mi-1-m}$, we conclude by \Cref{division-F(q)-dimq=2m+1} that $\ind(\alpha_{i})_{K(\varphi)}=\ind \alpha_{i}$. 
    Otherwise, $\ind \alpha_{i}\geq2^{me-mi}$, and hence $\ind(\alpha_{i})_{K(\varphi)}\geq2^{me-mi-1}$.
    Therefore we have $\ind(\alpha_{i})_{K(\varphi)}\geq2^{me-mi-1}$ for $0\leq i\leq e-1$.
    This shows that $K(\varphi)/F\in\mathcal{C}$.
    In view of the choice of $K$, this implies that $\varphi$ is isotropic.
    This argument shows that $u(K)\leq2m$.
\end{proof}

We can now show that the bound in \Cref{bounds-ind-u-inv} is optimal when \hbox{$u(F)\leq4$}.

\begin{ex}\label{examples-um-e-me-1}
    Let $m,e\in\nat^+$ with $m\geq2$. 
    By \cite[Construction 2.8]{SchBerg1992}, there exist a nonreal field $F$ of characteristic different from $2$ and a central $F$-division algebra $D$ such that $\exp D=2^e$, $\deg D=2^{me-1}$ and $\ind D^{\otimes2^{i}}=2^{me-1-i}$ for $1\leq i\leq e-1$. 
    Then \Cref{um-exp2^e-deg2^(me-1)} (applied to the the Brauer equivalence class of $D$) yields a field extension $F'/F$ such that $u(F')\leq2m$, $\exp D_{F'}=2^e$ and $\ind D_{F'}=2^{me-1}$.
    In the case where $m=2$, it follows that $u(F')=4$.
\end{ex}

 \begin{ex}\label{example-K/F-u6}
    By \Cref{examples-um-e-me-1}, there exist a nonreal field $F$ with $\cchar F\neq 2$ together with an $F$-division algebra $D$ of degree $8$ such that $u(F)=4$ and $D^{\otimes 4}$ is split.
    By \Cref{u-4-quadext-6-criterion}, it follows that $\ind D^{\otimes 2}=2$ and that $u(K)=6$ for every quadratic field extension $K/F$ such that $(D^{\otimes 2})_K$ is split.
\end{ex}

\section*{Acknowledgments}
We thank the referee for their comments.
This work was supported by the \emph{Fonds Wetenschappelijk Onderzoek – Vlaanderen} (\emph{FWO}) in the FWO Odysseus Programme (project G0E6114N `{Explicit Methods in Quadratic Form Theory}'),  the \emph{Fondazione Cariverona} in the programme Ricerca Scientifica di Eccellenza 2018 (project `Reducing complexity in algebra, logic, combinatorics - REDCOM'), and by the \emph{FWO-Tournesol programme} (project VS05018N).

\end{document}